\documentclass[12pt,reqno,a4paper]{article}
\usepackage[english]{babel}
\usepackage{amsmath}
\usepackage{amsfonts}
\usepackage{amsthm}
\usepackage{amssymb}

\usepackage{xspace}
\usepackage{euscript}
\usepackage{graphicx}
\usepackage{amscd}
\usepackage{tabularx}

\usepackage{enumerate}
 \usepackage{epsfig} 
 \usepackage{graphics} 
\theoremstyle{plain}
\newtheorem{theo}{Theorem}[section]
\newtheorem{lem}[theo]{Lemma}
\newtheorem{prop}[theo]{Proposition}
\theoremstyle{definition}
\newtheorem{definition}[theo]{Definition}
\theoremstyle{remark}
\newtheorem{rem}[theo]{Remark}
\numberwithin{equation}{section}

\newcommand{\C}{\mathbb{C}}

\newcommand{\R}{\mathbb{R}}
\newcommand{\N}{\mathbb{N}}
\newcommand{\M}{\mathbb{M}}
\newcommand{\divrg}{\textrm{div}\,}

\title{Stable determination of an inclusion in an inhomogeneous elastic body by boundary measurements
\thanks{The second author is supported by Universit\`a degli Studi di Trieste FRA 2012 `Problemi Inversi' and by GNAMPA of the
Istituto Nazionale di Alta Matematica (INdAM).} }
\author{Antonino
Morassi\thanks{Dipartimento di Ingegneria Civile e Architettura,
Universit\`a degli Studi di Udine, via Cotonificio 114, 33100
Udine, Italy. E-mail: \textsf{antonino.morassi@uniud.it}} \  and
Edi Rosset\thanks{Dipartimento di Matematica e Geoscienze,
Universit\`a degli Studi di Trieste, via Valerio 12/1, 34127
Trieste, Italy. E-mail: \textsf{rossedi@units.it}}}

\begin{document}

\maketitle

\begin{abstract}
In this paper we consider the stability issue for the inverse
problem of determining an unknown inclusion contained in an
elastic body by all the pairs of measurements of displacement and
traction taken at the boundary of the body. Both the body and the
inclusion are made by inhomogeneous linearly elastic isotropic
material. Under mild a priori assumptions about the smoothness of
the inclusion and the regularity of the coefficients, we show that
the logarithmic stability estimate proved in \cite{ADiCMR14} in
the case of piecewise constant coefficients continues to hold in
the inhomogeneous case. We introduce new arguments which allow to
simplify some technical aspects of the proof given in
\cite{ADiCMR14}.
\end{abstract}

\centerline{}

\section{Introduction}
\label{SecIntroduction}

The inverse problem of determining unknown inclusions in
continuous bodies {}from measurements of physical parameters taken
at the boundary of the body has attracted a lot of attention in
the last thirty years, see, among other contributions, the
reconstruction results obtained in \cite{Ike02}, \cite{UW08},
\cite{UWW09}. Inclusions may be due to the presence of
inhomogeneities or defects inside the body, and the development of
non-invasive testing approaches is of great importance in several
practical contexts, ranging {}from medicine to engineering
applications.

Inverse problems of this class are usually ill-posed according to
Hadamard's definition, and one of the main issues is the
uniqueness of the solution, that is the determination of the boundary
measurements which ensure the unique determination of the defect.
Moreover, {}from the point of view of practical applications, it
is crucial to establish how small perturbations on the data may
affect the accuracy of the identification of the inclusion,
namely, the study of the stability issue.

The prototype of these inverse problems is the determination of an
inclusion inside an electric conductor {}from boundary
measurements of electric potential and current flux. Uniqueness
was first proved by Isakov in $'88$ \cite{Is88}. The first
stability result is due to Alessandrini and Di Cristo
\cite{ADiC05}, who derived a logarithmic stability estimate of the
inclusion {}from all possible boundary measurements, that is
{}from the full Dirichlet-to-Neumann map. More precisely, the
authors considered in \cite{ADiC05} the case of piecewise-constant
coefficients and constructed an ingenious proof which, starting
{}from Alessandrini's identity (first derived in \cite{Al88}),
makes use of fundamental solutions for elliptic equations with
discontinuous coefficients, and suitable quantitative forms of
unique continuation for solutions to Laplacian equation. An
extension of the above result to the case of variable coefficients
was derived in \cite{DiC07}. The pioneering work \cite{ADiC05}
stimulated a subsequent line of research in which methods and
results were extended to other frameworks, such as, for example,
the stable identification of inclusions in thermal conductors
\cite{DiCV10}, \cite{DiCV11}, which involves a parabolic equation
with discontinuous coefficients.

Concerning the determination of an inclusion in an elastic body
{}from the Dirichlet-to-Neumann map, the uniqueness was proved by
Ikehata, Nakamura and Tanuma in \cite{INT99}. The stability issue
has been recently faced in \cite{ADiCMR14}. The statical
equilibrium of the defected body is governed by the following
system of elliptic equations
\begin{center}
\( {\displaystyle
\begin{array}{lr}
  \divrg((\C + (\C^D - \C)\chi_D)\nabla u)=0,
  & \hbox{in}\ \Omega,
\end{array}
 } \) \vskip -3.4em
\begin{eqnarray}
& & \label{eq:intro-1}
\end{eqnarray}
\end{center}
where $u$ is the three-dimensional displacement field inside the
elastic body $\Omega$, $\chi_D$ is the characteristic function of
the inclusion $D$, and $\C$, $\C^D$ is the elasticity tensor in
the background material and inside the inclusion, respectively.
Given inclusions $D_1$, $D_2$, let $\Lambda_{D_i}: H^{1/2}(\partial \Omega) \rightarrow
   H^{-1/2}(\partial \Omega)$ be the Dirichlet-to-Neumann map which
   gives the traction at the boundary $\partial \Omega$
   corresponding to a displacement field assigned on $\partial
   \Omega$, when $D=D_i$, $i=1,2$. Assuming that $\C$, $\C^{D_1}=\C^{D_2}$ are \textit{constant} and
   of Lam\'{e} type (e.g., isotropic material), and under
   $C^{1, \alpha}$-regularity of the boundary of the inclusion, the authors derived the following
   stability result. If, for some $\epsilon$, $0<\epsilon<1$,
\begin{equation}
  \label{eq:intro-2}
   \| \Lambda_{D_1}-\Lambda_{D_2}\|_{\mathcal{L}(H^{1/2}(\partial \Omega), H^{-1/2}(\partial
   \Omega))}\leq \epsilon,
\end{equation}
then the Hausdorff distance between the two inclusions can be
controlled as
\begin{equation}
  \label{eq:intro-3}
   d_H(\partial D_1, \partial D_2) \leq \frac{C}{| \log \epsilon |^\eta} ,
\end{equation}
where the constants $C>0$ and $\eta$, $0<\eta\leq 1$, only depend
on the a-priori data.

The piecewise-constant Lam\'{e} case can be considered as a
simplified mathematical model of real elastic bodies. Therefore,
it is of practical interest to extend the stability estimate
\eqref{eq:intro-3} to \textit{variable} coefficients both in the
background, $\C=\C(x)$, and in the inclusions,
$\C^{D_i}=\C^{D_i}(x)$, $i=1,2$. More precisely, assuming
$C^{1,1}$ and $C^\tau$ regularity, $\tau \in (0,1)$, for $\C$ and
$\C^{D_i}$, respectively, $i=1,2$, in this paper we show that
\eqref{eq:intro-3} continues to hold. Let us emphasize that in
order to derive our result the exact knowledge of the elasticity
tensor inside the inclusion is not needed.
In fact, only the strong convexity conditions \eqref{eq:3.1.new}
and the bounds \eqref{eq:3.2.new}, \eqref{eq:4.1.new},
\eqref{eq:4.3.new} are required. Moreover, as in \cite{ADiCMR14},
the inclusion is allowed to share a portion of its boundary with
the boundary of the body $\Omega$.

Let us briefly recall the main ideas of our approach and the new
mathematical tools we used in the proof of the stability result.
Let $\Gamma^{D_i}$ be the fundamental matrix
associated to the elasticity tensor $(\C +
(\C^{D_i}-\C)\chi_{D_i})$, $i=1,2$. The main idea is to obtain an
upper and a lower bound for $(\Gamma^{D_2}-\Gamma^{D_1})(y,w)$ for
points $y$ and $w$ belonging to the connected component of $\R^3
\setminus (\overline{D_1 \cup D_2})$ which contains $\R^3
\setminus \overline{\Omega}$, and approaching non-tangentially a
suitable point $P \in
\partial D_1 \setminus \overline{D_2}$ (or $
\partial D_2 \setminus \overline{D_1}$). A first crucial ingredient
in determining both upper and lower bounds is the integral
representation of $(\Gamma^{D_2}-\Gamma^{D_1})(y,w)$ given by
formula \eqref{integr-formula-D2_D1}. Next, the upper bound
follows {}from an application of Alessandrini's identity (suitably
adapted to linear elasticity, see Lemma $6.1$ in \cite{ADiCMR14})
and a propagation of smallness argument based on iterated use of
the three spheres inequality for solutions to the Lam\'{e} system
of linear elasticity with smooth variable coefficients.

In proving the lower bound (see Section \ref{proof-Lower_bound})
we introduce new arguments which entail a simplification of the
proof given for the piecewise-constant coefficient case. Indeed, a
generalization of Theorem $8.1$ in \cite{ADiCMR14}, which was a
key tool in proving the lower bound, should need the derivation of
an asymptotic approximation of $\Gamma^D$ in terms of the
fundamental matrix obtained by locally flattening the boundary
$\partial D$ and freezing the coefficients at a point belonging to
$\partial D$, which does not appear straightforward.

Finally, let us emphasize that the statement of Theorem $8.1$ in
\cite{ADiCMR14}, besides being worth of interest {}from a
theoretical viewpoint, may have relevant interest for its possible
applications. In fact, it turned out to be a fundamental
ingredient in the proof of Lipschitz stability estimates for the
inverse problem of determining the Lam\'e moduli for a piecewise
constant elasticity tensor corresponding to a known partition of
the body in a finite number of subdomains having regular
interfaces \cite{BFMRV14}, see also \cite{BFV13} for the case of
flat interfaces.

The plan of the paper is as follows. Notation and the a priori
information are introduced in section \ref{Main}, together with
the statement of the stability result (Theorem \ref{theo:6.1}). In
section \ref{Proof} we recall some auxiliary results, we state the
upper and lower bounds on $(\Gamma^{D_2}-\Gamma^{D_1})$, Theorems
\ref{theo:9.1} and \ref{theo:6.5}, and we give the proof of the
main Theorem \ref{theo:6.1}. Section \ref{proof-Lower_bound} is
devoted to the proof of Theorem \ref{theo:6.5}.

\section{The main result}
\label{Main}

\subsection{Notation}
\label{Main-notation}

Let us denote $\R^{3}_+= \{x\in \R^3 \ | \ x_3>0 \}$ and $\R^3_-=
\{x\in \R^3 \ | \ x_3<0 \}$. Given $x\in{\R}^3$, we shall denote
$x=(x',x_3)$, where $x'=(x_1,x_2)\in{\R}^{2}$, $x_3\in{\R}$. Given
$x \in \R^3$ and $r>0$, we shall use the following notation for
balls in three and two dimensions:
\begin{equation*}
    B_r(x)=\{y \in \R^3 \ \mid \ |y-x|<r\}, \quad B_r=B_r(O),
\end{equation*}
\begin{equation*}
    B'_r(x')=\{y' \in \R^2 \ \mid \ |y'-x'|<r\}, \quad
    B'_r=B'_r(O).
\end{equation*}

\begin{definition}
  \label{def:2.1} (${C}^{k,\alpha}$ regularity)
Let $E$ be a domain in ${\R}^{3}$. Given $k$,
$\alpha$, $k\in \N$, $0<\alpha\leq 1$, we say that $E$
is of \textit{class ${C}^{k,\alpha}$ with
constants $\rho_{0}$, $M_{0}>0$}, if, for any $P \in \partial E$, there
exists a rigid transformation of coordinates under which we have
$P=0$ and
\begin{equation*}
  E \cap B_{\rho_{0}}(O)=\{x \in B_{\rho_{0}}(O)\quad | \quad
x_{3}>\varphi(x')
  \},
\end{equation*}
where $\varphi$ is a ${C}^{k,\alpha}$ function on $B'_{\rho_{0}}$
satisfying
\begin{equation*}
\varphi(O)=0,
\end{equation*}
\begin{equation*}
|\nabla \varphi (O)|=0, \quad \hbox{when } k \geq 1,
\end{equation*}
\begin{equation*}
\|\varphi\|_{{C}^{k,\alpha}(B'_{\rho_{0}}(O))} \leq M_{0}\rho_{0}.
\end{equation*}

\end{definition}

Here and in the sequel all norms are normalized such that their
terms are dimensionally homogeneous. For instance
\begin{equation*}
  \|\varphi\|_{{C}^{k,\alpha}(B'_{\rho_{0}}(O))} =
  \sum_{i=0}^k \rho_0^i
  \|\nabla^i\varphi\|_{{L}^{\infty}(B'_{\rho_{0}}(O))}+
  \rho_0^{k+\alpha}|\nabla^k\varphi|_{\alpha, B'_{\rho_0}(O)},
\end{equation*}
where
\begin{equation*}
|\nabla^k\varphi|_{\alpha, B'_{\rho_0}(O)}= \sup_
{\overset{\scriptstyle x', \ y'\in B'_{\rho_0}(O)}{\scriptstyle
x'\neq y'}} \frac{|\nabla^k\varphi(x')-\nabla^k\varphi(y')|}
{|x'-y'|^\alpha}.
\end{equation*}

Similarly, for a vector function $u: \Omega \subset \R^3
\rightarrow \R^3$, we set
\begin{equation*}
\|u\|_{H^1(\Omega, \R^3)}=\left(\int_\Omega |u|^2
+\rho_0^2\int_\Omega|\nabla u|^2\right)^{\frac{1}{2}},
\end{equation*}
and so on for boundary and trace norms such as
$\|\cdot\|_{H^{\frac{1}{2}}(\partial\Omega, \R^3)}$,
$\|\cdot\|_{H^{-\frac{1}{2}}(\partial\Omega, \R^3)}$.

For any $U \subset \R^3$ and for any $r>0$, we denote
\begin{equation}
  \label{eq:2.int_env}
  U_{r}=\{x \in U \mid \textrm{dist}(x,\partial U)>r
  \},
\end{equation}
\begin{equation}
  \label{eq:2.ext_env}
  U^{r}=\{x \in \R^3 \mid \textrm{dist}(x, U )<r
  \}.
\end{equation}
We denote by $\M^{m\times n}$ the space of $m \times n$ real
valued matrices and we also use the notation $\M^n=\M^{n\times n}$.
Let ${\cal L} (X, Y)$ be the space of bounded
linear operators between Banach spaces $X$ and $Y$.

For every pair of real $n$-vectors $a$ and $b$, we denote by $a
\otimes b$ the $n \times n$ matrix with entries
\begin{equation}
  \label{eq:diade}
  (a \otimes b)_{ij} = a_i b_j, \quad i,j=1,...,n.
\end{equation}

For every $3 \times 3$ matrices $A$, $B$ and for every $\C\in{\cal
L} ({\M}^{3}, {\M}^{3})$, we use the following notation:
\begin{equation}
  \label{eq:2.notation_1}
  ({\C}A)_{ij} = \sum_{k,l=1}^{3} C_{ijkl}A_{kl},
\end{equation}
\begin{equation}
  \label{eq:2.notation_2}
  A \cdot B = \sum_{i,j=1}^{3} A_{ij}B_{ij},
\end{equation}
\begin{equation}
  \label{eq:2notation_3}
  |A|= (A \cdot A)^{\frac {1} {2}},
\end{equation}
where $C_{ijkl}$, $A_{ij}$ and $B_{ij}$ are the entries of $\C$,
$A$ and $B$ respectively.

Let us recall the definition of the Hausdorff distance $d_H(A,B)$ of two bounded closed sets $A,B\subset \R^3$
\begin{equation*}
  \label{eq:Hausdorff}
  d_H(A,B)= \max\left\{\max_{x\in A}d(x,B), \max_{x\in B}d(x,A)\right\}
\end{equation*}

\subsection{A-priori information and main result}
\label{Main-apriori-result}

We make the following a-priori assumptions. The continuous body
$\Omega$ is a bounded domain in $\R^3$ such that
\begin{equation}
  \label{eq:1.0}
  \R^3\setminus \overline{\Omega}\ \hbox{is connected},
\end{equation}
\begin{equation}
  \label{eq:1.1}
  |\Omega| \leq M_1 \rho_0^3,
\end{equation}
\begin{equation}
  \label{eq:1.1bis}
  \Omega \ \hbox{is of class } C^{1,\alpha}, \ \hbox{with
  constants } \ \rho_0, \ M_0,
\end{equation}
and the inclusion $D$ is a connected subset of $\Omega$ satisfying
\begin{equation}
  \label{eq:1.2}
  \R^3 \setminus \overline{D} \ \hbox{is connected},
\end{equation}
\begin{equation}
  \label{eq:1.4}
    D \ \hbox{is of class } C^{1,\alpha},\ \hbox{with
  constants } \ \rho_0, \ M_0,
\end{equation}
where $\rho_0$, $M_0$, $M_1$ are given positive constants, and
$0<\alpha \leq 1$.

The background material is linearly elastic isotropic, with
elasticity tensor $\C=\C(x)$, which - without restriction - may be
defined in the whole $\R^3$. The cartesian components of $\C(x)$
are
\begin{equation}
  \label{eq:2.1.new}
    C_{ijkl}(x)=\lambda(x) \delta_{ij}\delta_{kl} + \mu(x)
    (\delta_{ki}\delta_{lj}+\delta_{li}\delta_{kj}), \quad \hbox{for every } x \in
    \R^3,
\end{equation}
where $\delta_{ij}$ is the Kronecker's delta and the Lam\'{e}
moduli $\lambda=\lambda(x)$, $\mu=\mu(x)$ satisfy the strong
convexity conditions
\begin{equation}
  \label{eq:3.1.new}
    \mu(x) \geq \alpha_0, \quad 2\mu(x)+3\lambda(x) \geq \gamma_0,
    \quad \hbox{for every } x \in \R^3,
\end{equation}
for given constants $\alpha_0>0$, $\gamma_0 >0$. We shall also
assume upper bounds
\begin{equation}
  \label{eq:3.2.new}
    \mu(x) \leq \overline{\mu}, \quad \lambda(x) \leq
    \overline{\lambda},\quad \hbox{for every } x \in \R^3,
\end{equation}
where $\overline{\mu}>0$, $\overline{\lambda}\in \R$ are given
constants. Let us notice that \eqref{eq:2.1.new} clearly implies
the major and minor symmetries of $\C$, namely
\begin{equation}
  \label{eq:3.3.new}
    C_{ijkl}=C_{klij}=C_{lkij}, \quad \hbox{ }{i,j,k,l=1,2,3}.
\end{equation}
The inclusion $D$ is assumed to be made by linearly elastic
isotropic material having elasticity tensor $\C^D=\C^D(x)$ with
components
\begin{equation}
  \label{eq:3.4.new}
    C_{ijkl}^D(x)=\lambda^D(x) \delta_{ij}\delta_{kl} + \mu^D(x)
    (\delta_{ki}\delta_{lj}+\delta_{li}\delta_{kj}), \quad
    \hbox{for every } x \in \overline{\Omega},
\end{equation}
where the Lam\'{e} moduli $\lambda^D(x)$, $\mu^D(x)$ satisfy the
conditions \eqref{eq:3.1.new}--\eqref{eq:3.2.new} and, in
addition,
\begin{equation}
  \label{eq:4.1.new}
   (\lambda(x)-\lambda^D(x))^2+ (\mu(x)-\mu^D(x))^2\geq
   \eta_0^2>0,\quad \hbox{for every } x \in \overline{\Omega},
\end{equation}
for a given constant $\eta_0>0$.

Finally, the elasticity tensors $\C$ and $\C^D$ are assumed to be
of $C^{1,1}$ class in $\R^3$ and of $C^\tau$ class in $\overline{\Omega}$, $\tau \in
(0,1)$, respectively, that is
\begin{equation}
  \label{eq:4.2.new}
   \| \lambda \|_{C^{1,1}(\R^3)}+ \| \mu \|_{C^{1,1}(\R^3)} \leq
   M,
\end{equation}
\begin{equation}
  \label{eq:4.3.new}
   \| \lambda^D \|_{C^{\tau}(\overline{\Omega})}+ \| \mu^D \|_{C^{\tau}(\overline{\Omega})} \leq
   M,
\end{equation}
for a given constant $M>0$.

\medskip

For any $f \in H^{\frac{1}{2}}(\partial \Omega)$, let $u \in
H^1(\Omega)$ be the weak solution to the Dirichlet problem
\begin{center}
\( {\displaystyle \left\{
\begin{array}{lr}
  \divrg((\C + (\C^D - \C)\chi_D)\nabla u)=0,
  & \hbox{in}\ \Omega,
    \vspace{0.25em}\\
  u=f, & \hbox{on}\ \partial \Omega,
\end{array}
\right. } \) \vskip -4.4em
\begin{eqnarray}
& & \label{eq:5.1a.new}\\
& & \label{eq:5.1b.new}
\end{eqnarray}
\end{center}
where $\chi_D$ is the characteristic function of $D$. The
Dirichlet-to-Neumann map $\Lambda_D$ associated to
\eqref{eq:5.1a.new}--\eqref{eq:5.1b.new},
\begin{equation}
  \label{eq:5.2.new}
   \Lambda_D: H^{1/2}(\partial \Omega) \rightarrow
   H^{-1/2}(\partial \Omega),
\end{equation}
is defined in the weak form by
\begin{equation}
  \label{eq:5.3.new}
   <\Lambda_D f, v|_{\partial \Omega}>=\int_\Omega (\C + (\C^D - \C)\chi_D)
   \nabla u\cdot \nabla v,
\end{equation}
for every $v\in H^1(\Omega)$.

We prove the following logarithmic stability estimate for the
inverse problem of recovering the inclusion $D$ {}from the
knowledge of the map $\Lambda_D$.

\begin{theo}
   \label{theo:6.1}

Let $\Omega \subset \R^3$ be a bounded domain satisfying
\eqref{eq:1.0}--\eqref{eq:1.1bis} and let $D_1$, $D_2$ be two
connected inclusions contained in $\Omega$ satisfying
\eqref{eq:1.2}--\eqref{eq:1.4}. Let $\C(x)$ and $\C^{D_i}(x)$ be
the elasticity tensor of the material of $\Omega$ and of the
inclusion $D_i$, $i=1,2$, respectively, where $\C(x)$ given in
\eqref{eq:2.1.new} and $\C^{D_i}(x)$ given in \eqref{eq:3.4.new}
(for $D=D_i$) satisfy \eqref{eq:3.1.new}, \eqref{eq:3.2.new},
\eqref{eq:4.1.new}, \eqref{eq:4.2.new} and \eqref{eq:4.3.new}. If,
for some $\epsilon$, $0<\epsilon<1$,
\begin{equation}
  \label{eq:6.1.new}
   \| \Lambda_{D_1}-\Lambda_{D_2}\|_{\mathcal{L}(H^{1/2}(\partial \Omega), H^{-1/2}(\partial
   \Omega))}\leq \frac{\epsilon}{\rho_0},
\end{equation}
then
\begin{equation}
  \label{eq:6.2.new}
   d_H(\partial D_1, \partial D_2) \leq C  \rho_0  |\log \epsilon |^{-\eta},
\end{equation}
where $C>0$ and $\eta$, $0<\eta\leq 1$, are constants only
depending on $M_0$, $\alpha$, $M_1$, $\alpha_0$, $\gamma_0$, $
\overline{\mu}$, $ \overline{\lambda}$, $\eta_0$, $\tau$, $M$.
\end{theo}

\begin{rem}
If in Theorem \ref{theo:6.1} we further assume that the two
inclusions are at a prescribed distance {}from $\partial \Omega$,
then the result continues to hold even when the local
Dirichlet-to-Neumann map is known. The proof can be obtained by
adapting the general theory developed by Alessandrini and Kim
\cite{AK12}.
\end{rem}

\section{Proof of the main result}
\label{Proof}

In order to state the metric Lemma \ref{lem:7.1} below, we need to
introduce some notation.

We denote by $\mathcal{G}$ the connected component of $\R^3
\setminus (\overline{D_1 \cup D_2})$ which contains $\R^3
\setminus \overline{\Omega}$.

Given $O=(0,0,0)$, a unit vector $v$, $h>0$ and $\vartheta \in
\left ( 0, \frac{\pi}{2} \right )$, we denote by
\begin{equation}
  \label{eq:cono}
   C (O, v, h, \vartheta )
   =\left \{
   x \in \R^3 |\ |x - (x \cdot v)v| \leq \sin \vartheta |x|, \ 0 \leq x\cdot v \leq h  \right \}
\end{equation}
the closed truncated cone with vertex at $O$, axis along the
direction $v$, height $h$ and aperture $2\vartheta$. Given $R$,
$d$, $0 < R < d$ and $Q=-de_3 $, let us consider the cone $C \left
(O,-e_3, \frac{d^2-R^2}{d}, \arcsin \frac{R}{d}\right )$, whose
lateral boundary is tangent to the sphere
$\partial B_R(Q)$ along the circumference of its base.

Given a point $P\in\partial D_1\cap\partial \mathcal{G}$, let
$\nu$ be the outer unit normal to $\partial D_1$ at $P$ and let
$d>0$ be such that the segment $[P+d\nu,P]$ is contained in
$\overline{\mathcal{G}}$. For a point $P_0\in
\overline{\mathcal{G}}$, let $\gamma$ be a path in
$\overline{\mathcal{G}}$ joining $P_0$ to $P+d\nu$. We consider
the following neighbourhood of $\gamma\cup [P+d\nu,P]\setminus
\{P\}$ formed by a tubular neighbourhood of $\gamma$ attached to a
cone with vertex at $P$ and axis along $\nu$
\begin{equation}
  \label{eq:matita}
   V(\gamma,d,R) = \bigcup_{S \in \gamma} B_R(S) \cup
    C \left (P,\nu, \frac{d^2-R^2}{d}, \arcsin
   \frac{R}{d}\right ).
\end{equation}
Let us also define
\begin{equation}
  \label{eq:8.2}
   S_{2\rho_0} = \left \{ x\in\R^3 \ | \rho_0 < \hbox{dist}(x,\Omega) <2\rho_0 \right
   \}.
\end{equation}
\begin{lem}
\label{lem:7.1}
    Under the assumptions of Theorem \ref{theo:6.1}, up to inverting the role of $D_1$ and $D_2$, there exist
    positive constants $\overline{d}$, $\overline{c}$, where
    $\frac{\overline{d}}{\rho_0}$ only depends on $M_0$ and
    $\alpha$, and  $\overline{c} \geq 1$ only depends on $M_0$, $\alpha$ and $M_1$, and there exists
    a point $P \in \partial D_1  \cap \partial \mathcal{G} $ such that
\begin{equation}
  \label{eq:7.1.new}
        d_H (\partial D_1, \partial D_2) \leq \overline{c} \
        \hbox{dist}(P, D_2),
\end{equation}
and such that, giving any point $P_0 \in S_{2\rho_0}$, there
exists a path $\gamma \subset \Omega^{2\rho_0} \cap
\mathcal{G}$ joining $P_0$ to $P+\overline{d}\nu$, where
$\nu$ is the unit outer normal to $D_1$ at $P$, such that,
choosing a coordinate system with origin $O$ at $P$ and axis
$e_3=-\nu$, we have
\begin{equation}
  \label{eq:8BIS.4}
   V(\gamma, \overline{d}, \overline{R}) \subset \R^3 \cap \overline{\mathcal{G}},
\end{equation}
where $\frac{\overline{R}}{\rho_0}$
only depends on $M_0$ and
$\alpha$.
\end{lem}
The thesis of the above lemma is a straightforward consequence of
Lemma $4.1$ and Lemma $4.2$ in \cite{ADiCMR14}, and is inspired by
results obtained in \cite{AS12} and \cite{ADiC05}.

\medskip

Let $D$ be a domain of class $C^{1,\alpha}$ with constants
$\rho_0$, $M_0$ and $0<\alpha\leq 1$. The elasticity tensors $\C$
and $\C^D$ given by \eqref{eq:2.1.new} and \eqref{eq:3.4.new}
respectively, satisfy \eqref{eq:3.1.new}, \eqref{eq:3.2.new},
\eqref{eq:4.2.new} and \eqref{eq:4.3.new}.

Given $y \in \R^3$ and a \textit{concentrated force} $l
\delta(\cdot - y)$ applied at $y$, with $l \in \R^3$, let us
consider the \textit{normalized fundamental solution} $u^D \in
L^1_{loc}(\R^3, \R^3)$ defined by
\begin{equation}
  \label{eq:8.1.new}
  \left\{ \begin{array}{ll}
  \divrg_x \left ( (\C(x) + (\C^D(x) - \C(x))\chi_D)\nabla_x u^D(x,y;l) \right ) =-l\delta(x-y),
  & \hbox{in}\ \R^3\setminus \{y\},\\
  &  \\
      \lim_{|x| \rightarrow \infty} u^D(x,y;l)=0,\\
  \end{array}\right.
\end{equation}
where $\delta(\cdot - y)$ is the Dirac distribution supported at
$y$. It is well-known that
\begin{equation}
  \label{eq:8.2,new}
   u^D(x,y;l) = \Gamma^D(x,y)l,
\end{equation}
where $\Gamma^D=\Gamma^D(\cdot,y) \in L^1_{loc}(\R^3,
\mathcal{L}(\R^3,\R^3))$ is the \textit{normalized fundamental
matrix} for the operator $\divrg_x((\C(x) + (\C^D(x) -
\C(x))\chi_D)\nabla_x (\cdot))$. Existence of $\Gamma^D$ and
asymptotic estimates are stated in the following Proposition.
\begin{prop}
   \label{prop:8.1}
Under the above assumptions, there exists a unique fundamental
matrix $\Gamma^D(\cdot, y) \in C^0(\R^3\setminus \{y\})$, such that
\begin{equation}
  \label{eq:8.3.new}
   \Gamma^D(x,y) = (\Gamma^D(y,x))^T, \quad \hbox{for every } x\in
   \R^3, \ x \neq y,
\end{equation}
\begin{equation}
  \label{eq:8.4.new}
   |\Gamma^D(x,y)| \leq C |x-y|^{-1}, \quad \hbox{for every } x\in
   \R^3, \ x \neq y,
\end{equation}
\begin{equation}
  \label{eq:8.5.new}
   |\nabla_x \Gamma^D(x,y)| \leq C |x-y|^{-2}, \quad \hbox{for every } x\in
   \R^3, \ x \neq y,
\end{equation}
where the constant $C>0$ only depends on $M_0$, $\alpha$,
$\alpha_0$, $\gamma_0$, $\overline{\lambda}$, $\overline{\mu}$,
$\tau$, $M$.
\end{prop}
A proof of Proposition \ref{prop:8.1} follows by merging the
regularity results by Li and Nirenberg \cite{LN03} and the
analysis by Hofmann and Kim \cite{HK07}, see \cite{ADiCMR14} for
details.

\medskip

Let $D_i$, $i=1,2$, be a domain of class $C^{1,\alpha}$ with
constants $\rho_0$, $M_0$ and $0<\alpha\leq 1$, and consider the
elasticity tensors
\begin{equation}
  \label{eq:elasticity-tensors}
   \C^1 = \C \chi_{\R^3\setminus D_1} + \C^{D_1}\chi_{D_1}, \quad
   \C^2 = \C \chi_{\R^3\setminus D_2} + \C^{D_2}\chi_{D_2},
\end{equation}
where $\C^{D_1}$, $\C^{D_2}$ given in \eqref{eq:3.4.new} (with
$D=D_1$ and $D=D_2$, respectively) satisfy \eqref{eq:3.1.new},
\eqref{eq:3.2.new} and \eqref{eq:4.3.new}.

The following Proposition \ref{prop:integr-repres} states an
integral representation involving the normalized fundamental
matrices corresponding to inclusions $D_1$ and $D_2$. Similar
identities will be introduced in Section 4, in order to prove
Theorem \ref{theo:6.5}. Since these integral representations are
basic ingredients for our approach, we present here a proof of
Proposition \ref{prop:integr-repres}, which is more exhaustive
with respect to that given in \cite[Proof of Lemma 6.2]{ADiCMR14},
where some details were implied.
\begin{prop}
    \label{prop:integr-repres}
    Let $D_i$ and $\C^{D_i}$, $i=1,2$, satisfy the above assumptions.
    Then, for every $y$, $w \in \R^3$, $y \neq w$, and for every $l$, $m\in \R^3$ we have
\begin{multline}
  \label{integr-formula-D2_D1}
   \left(\Gamma^{D_2}- \Gamma^{D_1}\right)(y,w)m \cdot l = \\
    =\int_{\Omega} \C^1 \nabla \Gamma^{D_1}(\cdot,y)l \cdot \nabla
   \Gamma^{D_2}(\cdot,w)m - \int_{\Omega} \C^2 \nabla \Gamma^{D_1}(\cdot,y)l \cdot \nabla
   \Gamma^{D_2}(\cdot,w)m.
\end{multline}

\end{prop}
\begin{proof}
Formula \eqref{integr-formula-D2_D1} is obtained by subtracting the two following identities
\begin{equation}
  \label{integr-formula-gamma2}
   \int_{\R^3} \C^1 \nabla \Gamma^{D_1}(\cdot,y)l \cdot \nabla
   \Gamma^{D_2}(\cdot,w)m = \Gamma^{D_2}(y,w)m \cdot l,
\end{equation}
\begin{equation}
  \label{integr-formula-gamma1}
   \int_{\R^3} \C^2 \nabla \Gamma^{D_1}(\cdot,y)l \cdot \nabla
   \Gamma^{D_2}(\cdot,w)m = \Gamma^{D_1}(y,w)m \cdot l.
\end{equation}
To prove \eqref{integr-formula-gamma2},
let
\begin{equation}
  \label{eq:P1}
   \mathcal{H}= \{ f:\R^3 \rightarrow \R^3 \ | \ f \in C^0(\R^3,
   \R^3) \cap H^1(\R^3, \R^3), \ f\hbox{ with compact support} \}.
\end{equation}
By the weak formulation of \eqref{eq:8.1.new} (with $D=D_1$), we
have
\begin{equation}
  \label{eq:P2}
   \int_{\R^3} \C^1 \nabla \Gamma^{D_1} (\cdot, y) l \cdot \nabla
   \varphi = \varphi(y) \cdot l, \quad \hbox{for every } \varphi
   \in \mathcal{H}.
\end{equation}
Let $\epsilon >0$, $R >0$, with $ \epsilon \leq \frac{|w-y|}{2}$,
$R \geq 2\max \{|y|,|w| \}$,  and choose $\varphi \in \mathcal{H}$
such that $supp(\varphi) \subset B_{2R}(0)$ and
$\varphi|_{B_R(0)\setminus B_\epsilon(w) } \equiv
\Gamma^{D_2}(\cdot, w)m$. Then, \eqref{eq:P2} can be rewritten as
\begin{equation}
  \label{eq:P3}
    I_{\epsilon,R} + I_\epsilon + I_{R,2R} = \Gamma^{D_2}(y,w)m \cdot l,
\end{equation}
where
\begin{equation}
  \label{eq:P4}
   I_{\epsilon,R}= \int_{B_R(0)\setminus B_\epsilon(w)} \C^1 \nabla \Gamma^{D_1} (\cdot, y) l \cdot \nabla
   \Gamma^{D_2} (\cdot, w) m,
\end{equation}
\begin{equation}
  \label{eq:P5}
   I_{\epsilon}= \int_{B_\epsilon(w)} \C^1 \nabla \Gamma^{D_1} (\cdot, y) l \cdot \nabla
   \varphi ,
\end{equation}
\begin{equation}
  \label{eq:P6}
   I_{R,2R}= \int_{B_{2R}(0)\setminus B_R(0)} \C^1 \nabla \Gamma^{D_1} (\cdot, y) l \cdot \nabla
   \varphi .
\end{equation}
Integrating by parts on $B_\epsilon(w)$ and recalling that $y \in
\R^3 \setminus \overline{B}_\epsilon(w)$, we have
\begin{equation}
  \label{eq:P7}
   I_{\epsilon}= \int_{\partial B_\epsilon(w)} (\C^1 \nabla \Gamma^{D_1} (\cdot, y) l)\nu  \cdot
    \Gamma^{D_2} (\cdot, w) m.
\end{equation}
For every $x \in \partial B_\epsilon(w)$ and by our choice of
$\epsilon$, we have $|x-y| \geq |y-w|-|w-x|\geq \frac{|y-w|}{2}$. Therefore, by
\eqref{eq:8.4.new} and \eqref{eq:8.5.new}, we have
\begin{equation}
  \label{eq:P8}
   I_{\epsilon} \leq  C \int_{|x-w|=\epsilon} \frac{1}{|x-y|^2}
   \frac{1}{|x-w|}\leq \frac{C\epsilon}{|y-w|^2},
\end{equation}
where the constant $C>0$ only depends on $M_0$, $\alpha$,
$\alpha_0$, $\gamma_0$, $\overline{\lambda}$, $\overline{\mu}$,
$\tau$, $M$.

Analogously, integrating by parts in $B_{2R}(0)\setminus B_R(0)$
and recalling that $\varphi =0$ on $\partial B_{2R} (0)$ and $y
\in B_{ \frac{R}{2}  }(0)$, we have
\begin{equation}
  \label{eq:P9}
   I_{R,2R} = - \int_{\partial B_R(0)} (\C^1 \nabla \Gamma^{D_1} (\cdot, y) l)\nu  \cdot
   \Gamma^{D_2} (\cdot, w) m.
\end{equation}
For every $x \in \partial B_R(0)$ and by our choice of $R$, we
have $|x-w| \geq |x| - |w|\geq \frac{R}{2}$ and $|x-y| \geq \frac{R}{2}$.
Therefore,
\begin{equation}
  \label{eq:P10}
   I_{R,2R} \leq  C \int_{|x|=R} \frac{1}{|x-y|^2}
   \frac{1}{|x-w|}\leq \frac{C}{R},
\end{equation}
where the constant $C>0$ only depends on $M_0$, $\alpha$,
$\alpha_0$, $\gamma_0$, $\overline{\lambda}$, $\overline{\mu}$,
$\tau$, $M$.

Using the estimates \eqref{eq:P8} and \eqref{eq:P10} in
\eqref{eq:P3}, and taking the limit as $\epsilon \rightarrow 0$
and $R \rightarrow \infty$, we obtain
\eqref{integr-formula-gamma2}.
Symmetrically, we obtain
\begin{equation}
  \label{integr-formula-gamma1-new}
   \int_{\R^3} \C^2 \nabla \Gamma^{D_1}(\cdot,y)l \cdot \nabla
   \Gamma^{D_2}(\cdot,w)m = \Gamma^{D_1}(w,y)l \cdot m.
\end{equation}
By using \eqref{eq:8.3.new}, we obtain
\eqref{integr-formula-gamma1}.
\end{proof}

\medskip

Let $P$, $P \in \partial D_1$, be the point introduced in Lemma
\ref{lem:7.1}. In the following two theorems, we use a cartesian
coordinate system such that $P \equiv  O = (0,0,0)$ and $\nu =
-e_3$, where $\nu$ is the unit outer normal to $D_1$ at $P$.
\begin{theo} [Upper bound on $(\Gamma^{D_2}-\Gamma^{D_1})$]
   \label{theo:9.1}
   Under the notation of Lemma \ref{lem:7.1}, let
\begin{equation}
  \label{eq:9.1.new}
     y_h = P - h e_3,
\end{equation}
\begin{equation}
  \label{eq:9.2.new}
     w_h = P -\lambda_w h e_3, \quad 0<\lambda_w < 1,
\end{equation}
with
\begin{equation}
  \label{eq:9.3.new}
     0<h\leq \overline{h}\rho_0,
\end{equation}
where $\overline{h}$ only depends on $M_0$ and $\alpha$.

Then, for every $l$, $m\in\R^3$,
$|l|=|m|=1$, we have
\begin{equation}
  \label{eq:10.1.new}
     |(\Gamma^{D_2}-\Gamma^{D_1})(y_h, w_h) m \cdot l | \leq \frac{C}{\lambda_w h } \epsilon ^{ C_1 \left (
     \frac{h}{\rho_0}\right)^{C_2}},
\end{equation}
where the positive constants $C$, $C_1$ and $C_2$  only depend on $M_0$, $\alpha$,
$M_1$, $\alpha_0$, $\gamma_0$, $\overline{\lambda}$,
$\overline{\mu}$, $\tau$ and $M$.
\end{theo}
For the proof of the above result, we refer to \cite[Section
7]{ADiCMR14}. To give an idea of the role played by Proposition
\ref{prop:integr-repres} in proving estimate \eqref{eq:10.1.new},
let us recall Alessandrini's identity

\begin{equation}
  \label{eq:Alessandrini}
   \int_\Omega \C^1 \nabla u_1 \cdot \nabla u_2
   - \int_\Omega\C^2 \nabla u_1 \cdot \nabla
   u_2
   = <(\Lambda_{D_1} -
   \Lambda_{D_2})u_2,u_1>,
\end{equation}
which holds for every pair of solutions $u_i \in H^1(\Omega)$ to
\eqref{eq:intro-1} with $D=D_i$, $i=1,2$.

By choosing in the above identity $u_1(\cdot) =
\Gamma^{D_1}(\cdot,y)l$, $u_2(\cdot) = \Gamma^{D_2}(\cdot,w)m$ with
$y$, $w \in S_{2\rho_0}$,
the first member of \eqref{eq:Alessandrini} coincides with the second member of
\eqref{integr-formula-D2_D1}, so that,
recalling the asymptotic estimate
\eqref{eq:8.4.new} and the hypothesis \eqref{eq:6.1.new}, we obtain the following smallness estimate
\begin{equation}
  \label{eq:25.1}
   |(\Gamma^{D_2}- \Gamma^{D_1})(y,w)m\cdot l| \leq  C\frac{\epsilon}{\rho_0}, \quad\hbox{ for every }
    y, w \in S_{2\rho_0},
\end{equation}
where $C>0$ only depends on $M_0$, $\alpha$, $M_1$, $\alpha_0$,
$\gamma_0$, $\overline{\lambda}$, $\overline{\mu}$, $\tau$, $M$.

This first smallness estimate is then propagated up to the points $y_h$, $w_h$, with a technical
construction based on iterated application of the three spheres inequality.

\begin{theo}[Lower bound on the function $(\Gamma^{D_2}-\Gamma^{D_1})$]
   \label{theo:6.5}

Under the notation of Lemma \ref{lem:7.1}, let
\begin{equation}
  \label{eq:11.1.new}
     y_h = P -h e_3.
\end{equation}
For every $i=1,2,3$, there exists $\lambda_w \in \left\{
\frac{2}{3}, \frac{3}{4}, \frac{4}{5} \right\}$ and there exists
$\widetilde{h} \in \left (0, \frac{1}{2} \right )$ only depending
on $M_0$, $\alpha$, $\alpha_0$, $\gamma_0$, $\overline{\lambda}$,
$\overline{\mu}$, $\eta_0$, $\tau$, $M$, such that
\begin{equation}
  \label{eq:11.2.new}
     |(\Gamma^{D_2}-\Gamma^{D_1})(y_h, w_h) e_i \cdot e_i)| \geq \frac{C}{h}, \quad \hbox{for every } h, \
     0<h<\widetilde{h}\,dist(P,D_2),
\end{equation}
where
\begin{equation}
  \label{eq:11.3.new}
     w_h = P -\lambda_w h e_3,
\end{equation}
and $C>0$ only depends on $M_0$, $\alpha$, $M_1$, $\alpha_0$, $\gamma_0$,
$\overline{\lambda}$, $\overline{\mu}$, $\tau$, $M$ and $\eta_0$.
\end{theo}
The proof of this key result will be given in Section \ref{proof-Lower_bound}.

We are now in position to prove the main result of this paper.

\begin{proof} [Proof of Theorem \ref{theo:6.1}]
By the upper bound \eqref{eq:10.1.new}, with $l=m=e_i$ for
$i\in\{1,2,3\}$, and the lower bound \eqref{eq:11.2.new}, we have
\begin{equation}
  \label{eq:12.1.new}
   C \leq \epsilon ^{ C_1 \left (
     \frac{h}{\rho_0}\right)^{C_2}}, \quad \hbox{for every } h, \
     0<h\leq \min\{\overline{h}\rho_0, \widetilde{h}\,d(P,D_2)\}
\end{equation}
where $C, C_1, C_2$ only depend on $M_0$, $\alpha$, $M_1$, $\alpha_0$, $\gamma_0$,
$\overline{\lambda}$, $\overline{\mu}$, $\tau$, $M$ and $\eta_0$.
By our regularity assumptions on the domains, there exists $\widetilde{C}>0$, only depending on
$M_0$, $\alpha$, $M_1$, such that
\begin{equation}
  \label{diam}
d(P,D_2) \leq \hbox{diam}(\Omega)\leq \widetilde{C}\rho_0.
\end{equation}
Set $h^* = \min\left\{\frac{\overline{h}}{\widetilde{C}},
\widetilde{h}\right\}$. Then inequality \eqref{eq:12.1.new} holds
for every $h$ such that $h\leq h^* d(P,D_2)$, with $h^*$ only
depending on $M_0$, $\alpha$, $M_1$, $\alpha_0$, $\gamma_0$,
$\overline{\lambda}$, $\overline{\mu}$, $\tau$, $M$ and $\eta_0$.
Taking the logarithm in \eqref{eq:12.1.new} and recalling that
$\epsilon \in (0,1)$, we obtain
\begin{equation}
  \label{eq:12.2.new}
   h \leq C\rho_0 \left ( \frac{1}{ |\log \epsilon |  } \right
   )^{ \frac{1}{C_2} }, \quad \hbox{for every } h, \
     0<h\leq h^* d(P,D_2),
\end{equation}
In particular, choosing $h = h^* d(P,D_2)$, we have
\begin{equation}
  \label{eq:12.3.new}
   d(P,D_2) \leq C\rho_0 \left ( \frac{1}{ |\log \epsilon |  } \right
   )^{ \frac{1}{C_2} }.
\end{equation}
The thesis follows {}from Lemma \ref{lem:7.1}.
\end{proof}

\section{Proof of Theorem \ref{theo:6.5}}
\label{proof-Lower_bound}

Let us recall that we have chosen a cartesian coordinate system
with origin $P \equiv O$ and $e_3= - \nu$, where $\nu$ is the unit
outer normal to $D_1$ at $P$.

Let $\C_0 = \C(O)$ be the constant Lam\'{e} tensor, having
Lam\'{e} moduli $\lambda\equiv \lambda(O)$, $\mu\equiv \mu(O)$,
and let $\C^{D_1}_0 = \C^{D_1} (O)$ be the constant Lam\'{e}
tensor with Lam\'{e} moduli $\lambda\equiv \lambda^{D_1}(O)$,
$\mu\equiv \mu^{D_1}(O)$. Moreover, let us introduce the
elasticity tensors $\C^+_0 = \C_0 \chi_{\R^3_-} +
\C^{D_1}_0\chi_{\R^3_+}$, $\C^1_0 = \C_0 \chi_{\R^3\setminus D_1}
+ \C^{D_1}_0\chi_{D_1}$.

Let $\Gamma$, $\Gamma_0$, $\Gamma^+_0$, $\Gamma^{D_1}_0$ be the fundamental matrices associated to
the tensors $\C$, $\C_0$, $\C^+_0$, $\C^1_0$, respectively.

In the above notation, we may write, for every $m,l\in \R^3$, $|l|=|m|=1$,
\begin{multline}
  \label{1.1}
|(\Gamma^{D_2}-\Gamma^{D_1})(y_h,w_h)m\cdot l|\geq |(\Gamma^{+}_0-\Gamma_0)(y_h,w_h)m\cdot l|-
|(\Gamma^{D_2}-\Gamma)(y_h,w_h)m\cdot l|- \\-|(\Gamma-\Gamma_0)(y_h,w_h)m\cdot l|-
|(\Gamma^{+}_0-\Gamma^{D_1}_0)(y_h,w_h)m\cdot l|- \\-|(\Gamma^{D_1}_0-\Gamma^{D_1})(y_h,w_h)m\cdot l|.
\end{multline}
The following Lemma, which is a straightforward consequence of
Proposition 9.3 and formula $(9.11)$, derived in \cite{ADiCMR14},
gives a positive lower bound for the term
$|(\Gamma^{+}_0-\Gamma_0)(y_h,w_h)e_i\cdot e_i|$, $i=1,2,3$, for a
suitable $w_h$.

\begin{lem}
    \label{lem:1.1}
For every $i=1,2,3$,
there exists $\lambda_w \in \left\{ \frac{2}{3}, \frac{3}{4},
\frac{4}{5} \right\}$ such that
\begin{equation}
  \label{2.1}
     \left | (\Gamma^+_0(y_h,w_h)- \Gamma_0(y_h,w_h))e_i \cdot e_i \right |
     \geq
     \frac{\mathcal{C}}{h}, \quad \hbox{for every } h>0,
\end{equation}
where $\mathcal{C} > 0$ only depends on $\alpha_0$, $\gamma_0$,
$\overline{\lambda}$, $\overline{\mu}$, $\eta_0$.
\end{lem}
{}From now on, let $\lambda_w$ be chosen accordingly to the above
lemma and let $h \leq \frac{1}{2}\min \{d(P,D_2), \frac{\rho_0}{
\sqrt{1+M_0^2} } \}$.

\medskip
\noindent
\emph{Term $\Gamma^{D_2}-\Gamma$.}

Let us consider the vector valued function
\begin{equation}
  \label{1.2}
     v(x)= (\Gamma^{D_2}-\Gamma)(x,w_h)m.
\end{equation}

Let us set $\rho= d(P,D_2)$. Since $d(w_h,P) =\lambda_w h\leq
h\leq \frac{\rho}{2}$, we have that $d(w_h,D_2)\geq d(P,D_2)
-d(w_h,P)\geq \frac{\rho}{2}$. Therefore $v(x)$ is a solution to
the  Lam\'{e} system
\begin{equation}
  \label{2.2}
     \divrg _x(\C\nabla_xv(x))=0, \quad\hbox{in } B_{\frac{\rho}{2}}(w_h).
\end{equation}
By the regularity estimate
\begin{equation}
  \label{3.2}
     \sup_{B_{\frac{\rho}{4}}(w_h)}|v(x)|\leq \frac{C}{\rho^{\frac{3}{2}}}
        \left(\int_{B_{\frac{\rho}{2}}(w_h)}|v(x)|^2\right)^{\frac{1}{2}},
\end{equation}
with $C$ only depending on $\alpha_0$, $\gamma_0$, $\overline{\lambda}$, $\overline{\mu}$,
and by applying the asymptotic estimates \eqref{eq:8.4.new} to $\Gamma^{D_2}$ and $\Gamma$, it follows that
\begin{equation}
  \label{4.2}
     \sup_{B_{\frac{\rho}{4}}(w_h)}|v(x)|\leq \frac{C}{\rho},
\end{equation}
where $C>0$ only depends on $M_0$, $\alpha$, $\alpha_0$, $\gamma_0$, $\overline{\lambda}$,
$\overline{\mu}$, $\tau$, $M$.

Since $d(y_h,w_h)=(1-\lambda_w)h\leq \frac{h}{3}\leq \frac{\rho}{6}$,
$y_h\in B_{\frac{\rho}{4}}(w_h)$ and
\begin{equation}
  \label{5.2}
     \left|(\Gamma^{D_2}-\Gamma)(y_h,w_h)m\cdot l\right| = |v(y_h)\cdot l|\leq \frac{C}{\rho}
        = \frac{C}{d(P,D_2)},
\end{equation}
for every $l,m\in \R^3$, $|l| = |m| =1$, with $C$ only depending on $M_0$, $\alpha$, $\alpha_0$,
$\gamma_0$, $\overline{\lambda}$, $\overline{\mu}$, $\tau$, $M$.

\medskip
\noindent
\emph{Term $\Gamma^{D_1}_0-\Gamma^{D_1}$.}

By the same arguments seen in the proof of Proposition
\ref{prop:integr-repres}, we have that, for every $y,w \in \R^3$,
$y \neq w$, and for every $l,m \in \R^3$,
\begin{equation}
  \label{2.3}
   \int_{\R^3} \C^1 \nabla \Gamma^{D_1}(\cdot,y)l \cdot \nabla
   \Gamma^{D_1}_0(\cdot,w)m = \Gamma^{D_1}_0(y,w)m \cdot l,
\end{equation}
\begin{equation}
  \label{3.3}
   \int_{\R^3} \C^1_0 \nabla \Gamma^{D_1}(\cdot,y)l \cdot \nabla
   \Gamma^{D_1}_0(\cdot,w)m = \Gamma^{D_1}(y,w)m \cdot l.
\end{equation}
Choosing $y=y_h$ and $w=w_h$, we have
\begin{equation}
  \label{1.3}
   (\Gamma^{D_1}_0-\Gamma^{D_1})(y_h,w_h)m\cdot l =\int_{\R^3} (\C^1-\C^1_0) \nabla \Gamma^{D_1}(\cdot,y_h)l \cdot \nabla
   \Gamma^{D_1}_0(\cdot,w_h)m = J+J_0,
\end{equation}
with
\begin{equation}
  \label{1.3bis}
   J= \int_{D_1}  (\C^{D_1}-\C^{D_1}_0)\nabla \Gamma^{D_1}(\cdot,y_h)l \cdot \nabla
   \Gamma^{D_1}_0(\cdot,w_h)m,
\end{equation}
\begin{equation}
  \label{1.3ter}
   J_0= \int_{\R^3\setminus D_1} (\C-\C_0) \nabla \Gamma^{D_1}(\cdot,y_h)l \cdot \nabla
   \Gamma^{D_1}_0(\cdot,w_h)m.
\end{equation}

Let us estimate $J$. We have trivially
\begin{equation}
  \label{2.4}
   |J|\leq C(I_1+I_2),
\end{equation}
where $C>0$ only depends on $M_0$, $\alpha$, $\alpha_0$, $\gamma_0$, $\overline{\lambda}$,
$\overline{\mu}$, $\tau$, $M$ and
\begin{equation}
  \label{3.4}
   I_1=  \int_{|x|\geq \rho_0} \frac{|(\C^{D_1}-\C^{D_1}_0)(x)|}{|x-y_h|^2|x-w_h|^2},
\end{equation}
\begin{equation}
  \label{4.4}
   I_2=  \int_{|x|\leq \rho_0} \frac{|(\C^{D_1}-\C^{D_1}_0)(x)|}{|x-y_h|^2|x-w_h|^2}.
\end{equation}
Let us first estimate $I_1$. Since $h\leq \frac{\rho_0}{2}$ and
$|x|\geq \rho_0$, we have that $|x-y_h|\geq |x| - |y_h| = |x| - h
\geq \frac{|x|}{2}$ and similarly $|x-w_h|\geq \frac{|x|}{2}$, so
that
\begin{equation}
  \label{5.4}
   I_1 \leq C  \int_{|x|\geq \rho_0} \frac{1}{|x|^4} = \frac{C}{\rho_0},
\end{equation}
with $C$ only depending on $\overline{\lambda}$, $\overline{\mu}$.
To estimate $I_2$, we use the fact that
\begin{equation}
  \label{1.5}
   |(\C^{D_1}-\C^{D_1}_0)(x)| = |\C^{D_1}(x) - \C^{D_1}(O)| \leq \frac{C}{\rho_0^\tau}|x|^\tau,
\end{equation}
with $C$ only depending on $M$, so that
\begin{equation}
  \label{2.5}
  I_2\leq \frac{C}{\rho_0^\tau} (I'_2 + I''_2),
\end{equation}
where
\begin{equation}
  \label{3.5}
  I'_2 = \int_A \frac{|x|^\tau}{|x-y_h|^2|x-w_h|^2},
\end{equation}
\begin{equation}
  \label{4.5}
  I''_2 = \int_B \frac{|x|^\tau}{|x-y_h|^2|x-w_h|^2},
\end{equation}
with $A=\{|x|\leq \rho_0, |x|<6|y_h-w_h|\}$, $B=\{6|y_h-w_h|\leq
|x|\leq \rho_0\}$.

We perform the change of variables $x= |y_h-w_h|z$ in $I'_2$, obtaining
\begin{equation}
  \label{5.5}
  I'_2 \leq 6^\tau |y_h-w_h|^{\tau-1}\int_{|z|\leq 6}
    \left(z-\frac{y_h}{|y_h-w_h|}\right)^{-2}\left(z-\frac{w_h}{|y_h-w_h|}\right)^{-2}.
\end{equation}
Since the integral on the right hand side is bounded by an absolute constant, see
\cite[Chapter 2, Section 11]{Mir70}, we have that
\begin{equation}
  \label{6.5}
  I'_2 \leq C |y_h-w_h|^{\tau-1},
\end{equation}
with $C$ only depending on $\tau$.

For every $x\in B$, we have
\begin{equation}
  \label{1.6}
  |x|\geq 6|y_h-w_h| = 6h(1-\lambda_w)\geq \frac{6}{5}h,
\end{equation}
so that
\begin{equation}
  \label{2.6}
  |x|\leq |x-y_h|+|y_h| = |x-y_h| +h \leq |x-y_h| + \frac{5}{6}|x|.
\end{equation}
Hence
\begin{equation}
  \label{3.6}
  \frac{1}{6}|x|\leq |x-y_h|,
\end{equation}
and, similarly,
\begin{equation}
  \label{4.6}
  \frac{1}{6}|x|\leq |x-w_h|.
\end{equation}
By \eqref{3.6}--\eqref{4.6}, we have
\begin{equation}
  \label{5.6}
  I''_2\leq 6^4\int_B|x|^{\tau-4}\leq C \int_{6|y_h-w_h|}^{\rho_0} r^{\tau-2}dr\leq C|y_h-w_h|^{\tau-1},
\end{equation}
where $C$ is an absolute constant.

{}From \eqref{2.4}, \eqref{5.4}, \eqref{2.5}, \eqref{6.5}, \eqref{5.6} and noticing that
$|y_h-w_h| = h(1-\lambda_w)\geq \frac{h}{5}$, we have
\begin{equation}
  \label{6.6}
  |J| \leq \frac{C}{h} \left(\frac{h}{\rho_0}+
    \left(\frac{h}{\rho_0}\right)^\tau\right),
\end{equation}
where $C$ only depends on $M_0$, $\alpha$, $\alpha_0$, $\gamma_0$,
$\overline{\lambda}$, $\overline{\mu}$, $\tau$, $M$.

The term $J_0$ is estimated analogously with $\tau$ replaced by $1$, and therefore, by \eqref{1.3},
\begin{equation}
  \label{6.6bis}
  |(\Gamma^{D_1}_0-\Gamma^{D_1})(y_h,w_h)m\cdot l| \leq \frac{C}{h} \left(\frac{h}{\rho_0}+
    \left(\frac{h}{\rho_0}\right)^\tau\right),
\end{equation}
where $C$ only depends on $M_0$, $\alpha$, $\alpha_0$, $\gamma_0$,
$\overline{\lambda}$, $\overline{\mu}$, $\tau$, $M$.

\medskip
\noindent
\emph{Term $\Gamma^+_0-\Gamma^{D_1}_0$.}

Arguing similarly to the proof of Proposition
\ref{prop:integr-repres}, we have that, for every $y,w\in \R^3$,
$y\neq w$, and for every $l,m \in \R^3$,
\begin{multline}
  \label{1.7}
   (\Gamma^+_0-\Gamma^{D_1}_0)(y,w)m \cdot l = \int_{\R^3} (\C_0^{D_1}-\C_0)(\chi_{D_1} -\chi_{\R^3_+})\nabla \Gamma^{D_1}_0(\cdot,y)l \cdot \nabla \Gamma^+_0(\cdot,w)m =\\
=\int_{D_1\setminus\R^3_+}
(\C_0^{D_1}-\C_0)\nabla \Gamma^{D_1}_0(\cdot,y)l \cdot \nabla \Gamma^+_0(\cdot,w)m-\\
-\int_{\R^3_+\setminus D_1} (\C_0^{D_1}-\C_0)\nabla
\Gamma^{D_1}_0(\cdot,y)l \cdot \nabla \Gamma^+_0(\cdot,w)m.
\end{multline}
Therefore
\begin{equation}
  \label{2.7}
   |(\Gamma^+_0-\Gamma^{D_1}_0)(y_h,w_h)m \cdot l|\leq  C\int_{A\cup B}\frac{1}{|x-y_h|^2|x-w_h|^2},
\end{equation}
where
\begin{equation}
A=\left\{x\in (\R^3_+\setminus D_1) \cup (D_1 \setminus \R^3_+)\ |\ |x|\geq \frac{\rho_0}{\sqrt{1+M_0^2}}\right\},
\end{equation}
\begin{equation}
B=\left\{x\in (\R^3_+\setminus D_1) \cup (D_1 \setminus \R^3_+)\ |\ |x|\leq \frac{\rho_0}{\sqrt{1+M_0^2}}\right\},
\end{equation}
and $C$ only depends on $M_0$, $\alpha$, $\alpha_0$, $\gamma_0$,
$\overline{\lambda}$, $\overline{\mu}$, $\tau$, $M$. By our
hypotheses, $h\leq \frac{\rho_0}{2\sqrt{1+M_0^2}}$. Hence, for
every $x\in A$, $h\leq \frac{|x|}{2}$, $|x-y_h|\geq |x| -h\geq
\frac{|x|}{2}$, and similarly $|x-w_h|\geq \frac{|x|}{2}$, so that
\begin{equation}
  \label{1.8}
   \int_A \frac{1}{|x-y_h|^2|x-w_h|^2}\leq 16 \int_{|x|\geq \frac{\rho_0}{\sqrt{1+M_0^2}}}\frac{1}{|x|^4}= \frac{C}{\rho_0},
\end{equation}
with $C$ only depending on $M_0$.

By the local representation of the boundary of $D_1$ as a $C^{1,\alpha}$ graph, it follows that

\begin{equation}
B\subset \left\{x\in \R^3\ |\   |x'|\leq \frac{\rho_0}{\sqrt{1+M_0^2}}, |x_3|\leq \frac{M_0}{\rho_0^\alpha}|x'|^{1+\alpha}\right\}.
\end{equation}
By performing the change of variables $z = \frac{x}{h}$, we have
\begin{multline}
  \label{2.8}
   \int_B \frac{1}{|x-y_h|^2|x-w_h|^2}\leq \int_{|x'|\leq \frac{\rho_0}{\sqrt{1+M_0^2}}}dx'\int_{-\frac{M_0}{\rho_0^\alpha}|x'|^{1+\alpha}}^{\frac{M_0}{\rho_0^\alpha}|x'|^{1+\alpha}}
    \frac{1}{|x-y_h|^2|x-w_h|^2} dx_3 =\\
    =\frac{1}{h} \int_{|z'|\leq \frac{\rho_0}{h\sqrt{1+M_0^2}}}dz'\int_{-M_0\left(\frac{h}{\rho_0}\right)^\alpha|z'|^{1+\alpha}}^{M_0\left(\frac{h}{\rho_0}\right)^\alpha|z'|^{1+\alpha}}
    \frac{1}{|z+e_3|^2|z+\lambda_w e_3|^2} dz_3\leq\\
    \leq \frac{1}{h} \int_{\R^2}dz'\int_{-M_0\left(\frac{h}{\rho_0}\right)^\alpha|z'|^{1+\alpha}}^{M_0\left(\frac{h}{\rho_0}\right)^\alpha|z'|^{1+\alpha}}
    \frac{1}{|z+e_3|^2|z+\lambda_w e_3|^2} dz_3.
\end{multline}
Denoting
\begin{equation}
  \label{1.9}
   D(z) = \left(|z'|^2 + (z_3+1)^2\right) \left(|z'|^2 + (z_3+\lambda_w)^2\right),
\end{equation}
we have
\begin{equation}
  \label{2.9}
   \int_B \frac{1}{|x-y_h|^2|x-w_h|^2}\leq \frac{1}{h}(J_1+J_2),
\end{equation}
where
\begin{equation}
  \label{3.9}
   J_1= \int_{|z'|\leq\left(\frac{1}{3M_0}\right)^{\frac{1}{1+\alpha}}}dz'
    \int_{-M_0\left(\frac{h}{\rho_0}\right)^\alpha|z'|^{1+\alpha}}^{M_0\left(\frac{h}{\rho_0}\right)^\alpha|z'|^{1+\alpha}} \frac{1}{D(z)}dz_3,
\end{equation}
\begin{equation}
  \label{4.9}
   J_2=
   \int_{|z'|\geq\left(\frac{1}{3M_0}\right)^{\frac{1}{1+\alpha}}}dz'
    \int_{-M_0\left(\frac{h}{\rho_0}\right)^\alpha|z'|^{1+\alpha}}^{M_0\left(\frac{h}{\rho_0}\right)^\alpha|z'|^{1+\alpha}} \frac{1}{D(z)}dz_3.
\end{equation}
To estimate $J_1$, let us notice that, recalling $h\leq \frac{\rho_0}{2\sqrt{1+M_0^2}}$,
\begin{equation}
   |z_3+\lambda_w|\geq \lambda_w -|z_3|\geq \frac{2}{3} - M_0\left(\frac{h}{\rho_0}\right)^\alpha|z'|^{1+\alpha} \geq \frac{1}{3},
\end{equation}
and, a fortiori, $|z_3+1|\geq \frac{1}{3}$. Hence $D(z)\geq \frac{1}{3^4}$ and
\begin{equation}
  \label{5.9}
   J_1\leq 3^4 \int_{|z'|\leq\left(\frac{1}{3M_0}\right)^{\frac{1}{1+\alpha}}}
    2M_0\left(\frac{h}{\rho_0}\right)^\alpha|z'|^{1+\alpha} dz' = C \left(\frac{h}{\rho_0}\right)^\alpha,
\end{equation}
with $C$ only depending on $M_0$ and $\alpha$.

To estimate $J_2$ we use the trivial inequality $D(z)\geq |z'|^4$
when $\alpha < 1$, and $D(z) \geq C(M_0)|z'|^{ \frac{7}{2}  }$
when $\alpha = 1$, so obtaining
\begin{equation}
  \label{1.10}
   J_2\leq  C \left(\frac{h}{\rho_0}\right)^\alpha,
\end{equation}
with $C$ only depending on $M_0$ and $\alpha$.

By \eqref{2.7}, \eqref{1.8}, \eqref{2.9}, \eqref{5.9},
\eqref{1.10}, we have

\begin{equation}
  \label{2.10}
   |(\Gamma^+_0-\Gamma^{D_1}_0)(y_h,w_h)m\cdot l|\leq \frac{C}{h}\left(\frac{h}{\rho_0} +\left(\frac{h}{\rho_0}\right)^\alpha\right),
\end{equation}
with $C$ only depending on $M_0$, $\alpha$, $\alpha_0$,
$\gamma_0$, $\overline{\lambda}$, $\overline{\mu}$, $\tau$, $M$.

\medskip
\noindent
\emph{Term $\Gamma-\Gamma_0$.}

Similarly to the proof of Proposition \ref{prop:integr-repres}, we have that, for every $y$, $w\in \R^3$,
$y\neq w$,
\begin{equation}
  \label{3.10}
   (\Gamma_0-\Gamma)(y,w)m\cdot l
    =\int_{\R^3} (\C-\C_0)
    \nabla \Gamma(\cdot,y)l \cdot \nabla
   \Gamma_0(\cdot,w)m.
\end{equation}
{}From this identity, the arguments of the proof are similar to
those seen to estimate the addend $J_0$ in the expression of
$(\Gamma^{D_1}_0-\Gamma^{D_1})(y_h,w_h)l\cdot m$ given by
\eqref{1.3}, so that
\begin{equation}
  \label{4.10}
   |(\Gamma_0-\Gamma)(y_h,w_h)m\cdot l| \leq \frac{C}{h}\left(\frac{h}{\rho_0} \right),
\end{equation}
with $C$ only depending on $M_0$, $\alpha$, $\alpha_0$,
$\gamma_0$, $\overline{\lambda}$, $\overline{\mu}$, $\tau$, $M$.

\medskip
\noindent \emph{Conclusion.} Finally, {}from \eqref{1.1},
\eqref{2.1}, \eqref{5.2}, \eqref{6.6bis}, \eqref{2.10},
\eqref{4.10}, we have
\begin{multline}
  \label{1.11}
   |(\Gamma^{D_2}-\Gamma^{D_1})(y_h,w_h)e_i\cdot e_i| \geq \\
    \geq\frac{\mathcal{C}}{h}\left(1-C_1\frac{h}{d(P,D_2)}-C_2\frac{h}{\rho_0} -C_3\left(\frac{h}{\rho_0}\right)^\alpha - C_4\left(\frac{h}{\rho_0}\right)^\tau\right),
\end{multline}
with $C_i$, $i=1,...,4$, only depending on $M_0$, $\alpha$,
$\alpha_0$, $\gamma_0$, $\overline{\lambda}$, $\overline{\mu}$,
$\tau$, $M$ and $\mathcal{C}$  only depending on $\alpha_0$,
$\gamma_0$, $\overline{\lambda}$, $\overline{\mu}$ and $\eta_0$.
Let $h_1 = \min \{ \frac{1}{2}, \frac{1}{5C_1} \}$,
$h_2=\min\left\{\frac{1}{2 \sqrt{1+M_0^2}}, \frac{1}{5C_2},
\frac{1}{(5C_3)^{\frac{1}{\alpha}}},
\frac{1}{(5C_4)^{\frac{1}{\tau}} }
\right\}$. If $h\leq \min\{h_1d(P,D_2),h_2\rho_0\}$, then
\begin{equation}
  \label{1.12}
   |(\Gamma^{D_2}-\Gamma^{D_1})(y_h,w_h)e_i\cdot e_i| \geq \frac{\mathcal{C}}{5h},
\end{equation}
Let $\widetilde{h}=\min\left\{h_1, \frac{h_2}{\widetilde{C}}\right\}$, where $\widetilde{C}$ has been
introduced in \eqref{diam}. Then inequality \eqref{eq:11.2.new} holds for every $h$ such that
$h\leq \widetilde{h}\, d(P,D_2)$.

\end{document}